\documentclass[12pt]{amsart}
\usepackage{amsmath,latexsym,amssymb}
\usepackage{enumerate}
\usepackage{amsthm}
\usepackage{epsfig,graphicx}
\usepackage{tikz}
\usetikzlibrary{arrows,automata}
\usepackage{mathrsfs}
\usepackage{tikz-3dplot}
\usetikzlibrary{3d}
\usepackage{float}
\usepackage{comment}

\usepackage[backend=biber]{biblatex}
\usepackage{shortcut}

\newtheorem{theorem}{Theorem}[section]
\newtheorem{definition}[theorem]{Definition}

\newtheorem{lemma}[theorem]{Lemma}
\newtheorem{remark}[theorem]{Remark}
\newtheorem*{remarks*}{Remarks}
\newtheorem{proposition}[theorem]{Proposition}

\newtheorem{algorithm}[theorem]{Algorithm}

\newcommand\RR{\mathbb{R}}
\newcommand\ZZ{\mathbb{Z}}

\title[Hausdorff measure of cartesian of Cantor sets]{Hausdorff measure of cartesian product of Cantor sets}

\author{Siyuan Guo}
\address[S. Guo]{Mathematics Department, Johns Hopkins University, 3400 North Charles St, Baltimore, MD 21218-2683, U.S.A}
\email{sguo45@jhu.edu}

\author[Taylor Jones]{Taylor Jones$^*$}
\address[T. Jones]{Applied Mathematics \& Statistics Department, Johns Hopkins University, 3400 North Charles St, Baltimore, MD 21218-2683, U.S.A.}
\email{rjone197@jh.edu}
\thanks{$^*$ Corresponding author.}

\begin{document}
	
	
	\begin{abstract}
		Hausdorff measure and Hausdorff dimension are useful tools to describe fractals. This paper investigates the bounds on the $d\log_32$-dimensional Hausdorff measure of the $d$-fold Cartesian product of the $1/3$ Cantor set, $\mathcal C^d$. By applying known theorems on the Hausdorff measure of fractals satisfying the strong open set condition and generalizing what has been done on $\mathcal C^2$, we compute stricter upper and lower bounds for the Hausdorff measure of $\mathcal C^d$ for several small integers $d$.
	\end{abstract}
	
	\maketitle
	
	
	\section{Introduction}
	
	The ternary Cantor set (denoted $\CC$) is commonly studied in many areas of mathematics. The set is famous for its constructions and properties. A common construction of the ternary Cantor set is given as follows.
	
	\begin{definition}[Ternary Cantor Set]
		\label{def: cantor middle thirds}
		Let
		\begin{equation*}
			C_0 = [0,1],
		\end{equation*}
		and for $n\ge1$ let
		\begin{equation*}
			C_n = \frac{C_{n-1}}{3}\cup\left(\frac{2}{3} + \frac{C_{n-1}}{3}\right) := \left\{\frac{x}{3}:x\in C_{n-1}\right\}\cup \left\{\frac{2}{3} + \frac{x}{3} : x\in C_{n-1}\right\}.
		\end{equation*}
		The Cantor set is defined as $\CC := \bigcap_{n=0}^\infty C_n$.
	\end{definition}
	
	\begin{figure}[H]
		\centering
		\begin{tikzpicture}
			\draw[thick] (-3,1.5) -- (3,1.5);
			\draw[|-|] (-3,1.5) -- (-1,1.5);
			\draw[|-|] (1,1.5) -- (3,1.5);
			\node at (-3.75,1.5) {$C_0$};
			\draw[thick] (-3,1) -- (-1,1) (1,1) -- (3,1);
			\draw[|-|] (-3,1) -- (-2.33,1);
			\draw[|-|] (-1.67,1) -- (-1,1);
			\draw[|-|] (3,1) -- (2.33,1);
			\draw[|-|] (1.67,1) -- (1,1);
			\node at (-3.75,1) {$C_1$};
			\draw[thick] (-3,.5) -- (-2.33,.5) (-1.67,.5) -- (-1,.5) (1,.5) -- (1.67,.5) (2.33,.5) -- (3,.5);
			\draw[|-|] (-3,0.5) -- (-2.78,0.5);
			\draw[|-|] (-2.56,0.5) -- (-2.33,0.5);
			\draw[|-|] (-1.67,0.5) -- (-1.44,0.5);
			\draw[|-|] (-1.22,0.5) -- (-1,0.5);
			\draw[|-|] (3,0.5) -- (2.78,0.5);
			\draw[|-|] (2.56,0.5) -- (2.33,0.5);
			\draw[|-|] (1.67,0.5) -- (1.44,0.5);
			\draw[|-|] (1.22,0.5) -- (1,0.5);
			\node at (-3.75,0.5) {$C_2$};
			\draw[thick] (-3,0) -- (-2.78,0) (-2.56,0) -- (-2.33,0) (-1.67,0) -- (-1.44,0) (-1.22,0) -- (-1,0) (3,0) -- (2.78,0) (2.56,0) -- (2.33,0) (1.67,0) -- (1.44,0) (1.22,0) -- (1,0);
			\draw[|-|] (-3,0) -- (-2.78,0);
			\draw[|-|] (-2.56,0) -- (-2.33,0);
			\draw[|-|] (-1.67,0) -- (-1.44,0);
			\draw[|-|] (-1.22,0) -- (-1,0);
			\draw[|-|] (3,0) -- (2.78,0);
			\draw[|-|] (2.56,0) -- (2.33,0);
			\draw[|-|] (1.67,0) -- (1.44,0);
			\draw[|-|] (1.22,0) -- (1,0);
			\node at (-3.75,0) {$C_3$};
			\node at (-3.75,-0.5) {$\vdots$};
			\node at (-2,-0.5) {$\vdots$};
			\node at (0,-0.5) {$\vdots$};
			\node at (2,-0.5) {$\vdots$};
		\end{tikzpicture}
		\caption{The Ternary Cantor set $\CC$, drawn to stage 3.}
		\label{fig: cator middle thirds set}
	\end{figure}
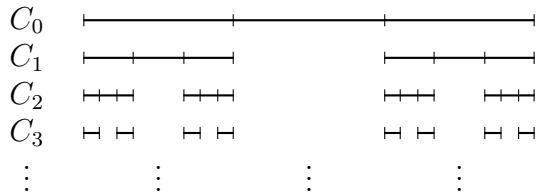
	
	The ternary Cantor set is particularly important in the study of Fractal Geometry, since the set itself is a ``fractal.'' In terms of the Lebesgue measure defined over $\R$, the set has measure $0$ since the total removed measure of the set through each level construction sums up to 1, so the set has measure 0 by the monotonicity that $\bigcap_{i=1}^NC_i\searrow \CC$.\\
    
    As a null sets in terms of Lebesgue measure on its dimension, a feature to characterize this set further is by its $s$-dimensional Hausdorff measure.
    
	\begin{definition}[$s$-dimensional Hausdorff Measure]
		Let $s\in \R^+$ be fixed. For any set $F\subset\R^d$, we define the $s$-dimensional Hausdorff measure as the limit:
		\begin{equation*}
			\CH^s(F) := \lim_{\delta\to 0}\left[\inf\left\{\sum_{i=1}^\infty|U_i|^s:|U_i| \leq \delta \mbox{ for all } i \mbox{ and } \bigcup_{i=1}^\infty U_i \supset F\right\}\right].
		\end{equation*}
		Note that the limit always exist or is $+\infty$, because the infimum monotonically increases as $\delta\to 0$, since any collection of sets $\{U_i\}_{i = 1}^\infty$ satisfying a smaller $\delta$ naturally satisfies a larger $\delta$.
	\end{definition}
	
	With this definition, the Hausdorff measures with respect to different choices of $s$ follow a distinct pattern.
	
	\begin{proposition}[Section 2.2 in~\cite{Falconer}]
		Suppose $F\subset\R^n$ and $\CH^s(F) < \infty$. If $t > s$, then $\CH^t(F) = 0$.
	\end{proposition}
	
	Thus, the graph of $\CH^s(F)$ against $s$ has some ``jump'' discontinuity from $\infty$ to $0$, as there could only be at most one $s$ with \textit{non-zero and finite} Hausdorff measure, which directly leads to the definition of the Hausdorff dimension.
	
	\begin{definition}[Hausdorff Dimension]
		Let $F\subset\R^n$, the Hausdorff dimension, denoted $\dim_{\rm H}F$, is defined as:
		\begin{equation*}
			\dim_{\rm H}F = \inf\{s\geq 0 : \CH^s(F) = 0\} = \sup\{s : \CH^s(F) = \infty\}.
		\end{equation*}
	\end{definition}

    When this pattern is plotted on a graph, we will clearly notice the jump at $\dim_{\rm H}F$. Although we do not know the exact value at of $\CH^{\dim_{\rm H}F}F$.

    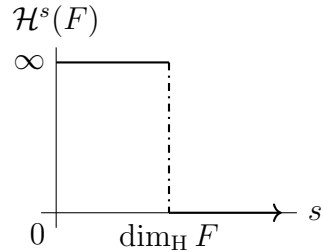
\begin{figure}[h]
        \centering
        \begin{tikzpicture}
            \draw (0,0) node[anchor=north east] {$0$};
            \draw (0,2) node[anchor=east] {$\infty$};
            \draw (1.5,0) node[anchor=north] {$\dim_{\rm H}F$};
            \draw (-0.2,0) -- (3.2,0) node[right] {$s$};
            \draw (0,-0.2) -- (0,2.2) node[above] {$\CH^s(F)$};
            \draw[thick] (0, 2) -- (1.5, 2);
            \draw[thick, dash dot] (1.5, 2) -- (1.5, 0);
            \draw[thick, ->] (1.5, 0) -- (3, 0);
        \end{tikzpicture}
        \caption{Graph of $\CH^s(F)$ against $s$ with jump at its dimension}
    \end{figure}
	
	For the case of the ternary Cantor set, its Hausdorff dimension is known to be $\log_32$, and it was shown by Felix Hausdorff himself in 1919~\cite{Hausdorff1919} (in German) that the exact Hausdorff measure of the ternary Cantor set is $\CH^{\log_32}(\CC) = 1$, in which he denotes it as $L_\rho = 1$ in the paper.\\

    While the Hausdorff measure of the ternary Cantor set and many fractals of Hausdorff dimension less than 1 is well studied (as explain in~\cite{MESURESDE}), there have been limited amount of study conducted on fractals with higher Hausdorff dimension ($\geq 1$).\\
	
	Particularly, the Cartesian products of the Cantor sets (denoted $\CC^d$ for the $d$-th product) remains not well studied. Specifically, we have the Cartesian product defined as:
	\begin{equation*}
		\CC^d := \{(x_1,\cdots, x_d)\in\R^d:x_i\in\CC\mbox{ for all }1\leq i \leq d\}.
	\end{equation*}
	
	Here, we can interpret the products of Cantor sets as a self-similar set generated by an Iterated Function System (IFS), composed of various contraction mappings $\{S_1, S_2,\cdots, S_{2^d}\}$ which satisfies the open set condition (defined below), where $S_i$ is a contraction scaled by $\frac{1}{3}$ into different segments, which can be considered a self-similar set.\\
    
	
	\begin{definition}
		We explicitly define the functions $S_1,\dots,S_{2^d}:\R^{d}\to\R^{d}$, as follows:
		\begin{enumerate}
			\item Let
			\[V_d:=\{\textbf{i}=(i_1,\dots,i_d)\in\RR^d:i_k\in\{0,1\}\:\forall\:k=1,2,\dots,d\}.\]
			For each $\textbf{i}\in V_d$ let $\kappa(\textbf{i})$ be the lexicographical position of $\textbf{i}$. E.g. $\kappa(0,0\dots,0)=1$, $\kappa(0,0\dots,1)=2,$ $\kappa(0,0,\dots,1,0)=3$, and so-on until $\kappa(1,1,\dots,1)=2^d$.
			\item  For each $i\in\{1,\dots,2^d\}$ let $\textbf{i}=\kappa^{-1}(i)$ and for any $d\times1$ vector $\textbf{x}$ define
			\[S_{i}(\textbf{x}):=\frac{1}{3}I_d\textbf{x}+\frac{2}{3}\textbf{i}^T\]
			where $I_d$ is the $d\times d$ identity matrix.
		\end{enumerate}
	\end{definition}

    So 
    \[\CC^d = \left\{\lim_{n\to\infty}S_{i_1}\circ S_{i_2}\circ\cdots\circ S_{i_n}(0):i_k\in\{1,2,\dots,2^d\}\:\text{for all }\: k\right\}\]
	
	Meanwhile, we provide a labeling for the $S_i([0,1]^3)$ with $i = 1,2,\cdots, 8$ in $\CC^3$:
	
	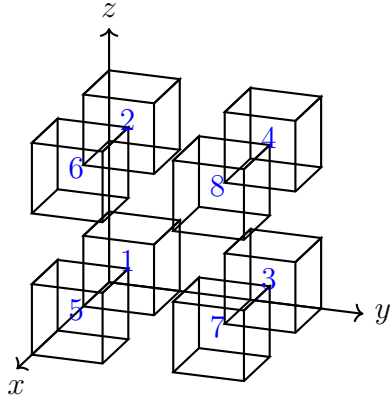
\begin{figure}[h]
		\begin{tikzpicture}[scale=3]
			\tdplotsetmaincoords{70}{110}
			
			\begin{scope}[tdplot_main_coords]
				\draw[->, thick] (0, 0, 0) -- (1.2, 0, 0) node[anchor=north] {$x$};
				\draw[->, thick] (0, 0, 0) -- (0, 1.2, 0) node[anchor=west] {$y$};
				\draw[->, thick] (0, 0, 0) -- (0, 0, 1.2) node[anchor=south] {$z$};
				
				\foreach \x in {0, 2/3} {
					\foreach \y in {0, 2/3} {
						\foreach \z in {0, 2/3} {
							\coordinate (A) at (\x, \y, \z);
							\coordinate (B) at (\x + 1/3, \y, \z);
							\coordinate (C) at (\x + 1/3, \y + 1/3, \z);
							\coordinate (D) at (\x, \y + 1/3, \z);
							\coordinate (E) at (\x, \y, \z + 1/3);
							\coordinate (F) at (\x + 1/3, \y, \z + 1/3);
							\coordinate (G) at (\x + 1/3, \y + 1/3, \z + 1/3);
							\coordinate (H) at (\x, \y + 1/3, \z + 1/3);
							
							\draw[thick] (A) -- (B) -- (C) -- (D) -- cycle;
							\draw[thick] (E) -- (F) -- (G) -- (H) -- cycle;
							\draw[thick] (A) -- (E) -- (F) -- (B);
							\draw[thick] (C) -- (G) -- (H) -- (D);
						}
					}
				}
				
				\foreach \x in {1, 2, 3, 4, 5, 6, 7, 8} {
					\pgfmathparse{int((\x - 1) / 4)}
					\let\firstcoord\pgfmathresult 
					\pgfmathparse{mod(int((\x - 1) / 2), 2)}
					\let\secondcoord\pgfmathresult 
					\pgfmathparse{mod((\x - 1), 2)}
					\let\thirdcoord\pgfmathresult 
					\draw[blue] (\firstcoord * 2 / 3, \secondcoord * 2 / 3, \thirdcoord * 2 / 3) node[anchor=south west] {\x};
				}
			\end{scope}
		\end{tikzpicture}
		\caption{Labeling of the Quadrants in level 1 basic sets of $\CC^3$.}
	\end{figure}

    While the dimensions are well studied, there have been some criterion for us to quickly conclude the dimension of some IFS.

    \begin{definition}[Open Set Condition]
    	Let $S_i:D\to D$ be contractions in IFS $\{S_i\}_{i=1}^m$, then $\{S_i\}$ satisfy the open set condition if there exists a non-empty bounded open set $V$ such that:
    	\begin{equation*}
    		V\supset \bigcup_{i=1}^mS_i(V).
    	\end{equation*}
    \end{definition}
    
    Note that $\CC^d$ satisfies the open set condition with $V = (0,1)^d$. There is a well known formula for computing the Hausdorff dimension of sets generated by an IFS of similarities meeting the open set condition, from Theorem 9.3 in \cite{Falconer}.

    \begin{theorem}
        Suppose that open set condition holds for the contractions $S_i$ on $\R^n$ with ratios $0 < c_i < 1$ for $1\leq i\leq m$. If $F$ is the attractor of the IFS $\{S_1,\cdots,S_m\}$, \ie:
	\begin{equation*}
		F = \bigcup_{i=1}^m S_i(F),
	\end{equation*}
	then $\dim_{\rm H}F=\dim_{\rm B}F = s$, where $s$ is given by:
	\begin{equation*}
		\sum_{i=1}^{m}c_i^s = 1.
	\end{equation*}
	Moreover, for this value of $s$, $0 < \CH^s(F) < \infty$.
    \end{theorem}
    
    Hence, the $d$-th product of the Cantor set shall be the attractor of the IFS composed of $2^d$ mappings, each scaled by $\frac{1}{3}$, so we correspondingly have the Hausdorff dimension as:
	\begin{equation}
		\label{Hausdorff dim d}
		s_d := \dim_{\rm H}(\CC^d) = d\log_3(2),
	\end{equation}
	and we have the trivial bounds directly from the theorem:
	\begin{equation}
		0 < \CH^{s_d}(\CC^d) < \infty. \label{eq::trivial-bd}
	\end{equation}
	
	The proof of Hausdorff measure of 1 dimensional ternary Cantor set being 1, \ie, $\CH^s(\CC) = 1$, can be found on various texts such as~\cite{FalconerII}, but the measure of higher dimensional ternary Cantor sets remain unknown.\\
	
	Currently, the best approximation of the 2 dimensional Cartesian product of cantor sets is achieved in~\cite{Deng} in 2012:
	\begin{align*}
		1.48329 \leq \CH^{s_2}(\CC^2) \leq 1.500886.
	\end{align*}

    Moreover, there have not been any study on the higher products of the Cantor sets.\\
    
	The goal of this paper is to extend the results of~\cite{Deng} to higher dimensional Cartesian products of the Cantor set, namely, we would want to think about the upper bound and lower bounded of the measure of $\CC^d$ that is better than the na\"ive bounds for the Cantor sets.\\
	
	In particular, the na\"ive upper bound can be derived from the definition of the Hausdorff measure for all Cantor sets $\CC^d$. Consider the level $k$ construction of the product of the Cantor Set, there will be, respectively, $(2^d)^k = 2^{d\cdot k}$ products of subintervals, in which the intervals has length $3^{-k}$, so the corresponding diameter of each cover is $\sqrt{d\cdot 3^{-2k}} = \sqrt{d}\cdot3^{-k}$. Hence, at each level $k$, the $s$-dimensional Hausdorff measure with the covering set diameters no less than $\sqrt{d}\cdot 3^{-k}$ is:
	\begin{equation*}
		2^{d\cdot k}\cdot (\sqrt{d}\cdot 3^{-k})^s = 2^{dk} \cdot d^{d\log_3(2)/2}\cdot 3^{-kd\log_3(2)} = \left(\frac{2}{2}\right)^{kd}\cdot d^{d\log_3(2)/2} = d^{d\log_3(2)/2}.
	\end{equation*}
	Note that as we get deeper with the layers, $k\to\infty$ and $\delta\to 0$, and because of the definition of infimum, we naturally have the upper bound of the Hausdorff measure at dimension $s$ for the $d$-th Cartesian product of the Cantor set as $d^{d\log_3(2)/2}$, \ie:
	\begin{equation}
		\CH^s(\CC^d) \leq d^{d\log_3(2)/2}. \label{eq::naive-upper}
	\end{equation}
	
	A few upper bounds can be numerically computed.
	
	\begin{figure}[H]
		\begin{center}
			\begin{tikzpicture}[x=1.5cm, y=0.2cm]
				\draw[->, line width=0.5pt] (0,0) -- (7,0) node[right] {$d$};
				\draw[->, line width=0.5pt] (0,0) -- (0,32) node[above] {$y$};
				
				\draw[color=blue, variable=\x, samples=100, domain=1:6, dashed] plot ({\x}, {\x^( (\x * ln(2)/ln(3)) / 2 )});
				\footnotesize
				\draw[fill=blue] (1,1) circle (1pt) node[above] {$(1,1)$};
				\draw[fill=blue] (2,1.54856265263) circle (1pt) node[above] {$(2,1.5485)$};
				\draw[fill=blue] (3,2.82842712475) circle (1pt) node[above] {$(3,2.8284)$};
				\draw[fill=blue] (4,5.75062600477) circle (1pt) node[above] {$(4,5.7506)$};
				\draw[fill=blue] (5,12.6620035654) circle (1pt) node[above] {$(5,12.6620)$};
				\draw[fill=blue] (6,29.7081993809) circle (1pt) node[above] {$(6,29.7081)$};
				
				\draw[fill=black] (1,0) circle (1pt) node[below] {1};
				\draw[fill=black] (2,0) circle (1pt) node[below] {2};
				\draw[fill=black] (3,0) circle (1pt) node[below] {3};
				\draw[fill=black] (4,0) circle (1pt) node[below] {4};
				\draw[fill=black] (5,0) circle (1pt) node[below] {5};
				\draw[fill=black] (6,0) circle (1pt) node[below] {6};
				
				\draw[fill=black] (0,5) circle (1pt) node[left] {5};
				\draw[fill=black] (0,10) circle (1pt) node[left] {10};
				\draw[fill=black] (0,15) circle (1pt) node[left] {15};
				\draw[fill=black] (0,20) circle (1pt) node[left] {20};
				\draw[fill=black] (0,25) circle (1pt) node[left] {25};
				\draw[fill=black] (0,30) circle (1pt) node[left] {30};
			\end{tikzpicture}
		\end{center}
		\caption{Na\"ive Upper Bound for the Hausdorff Measure for $\CC^d$ with $d = 1,\cdots, 6$.}
	\end{figure}
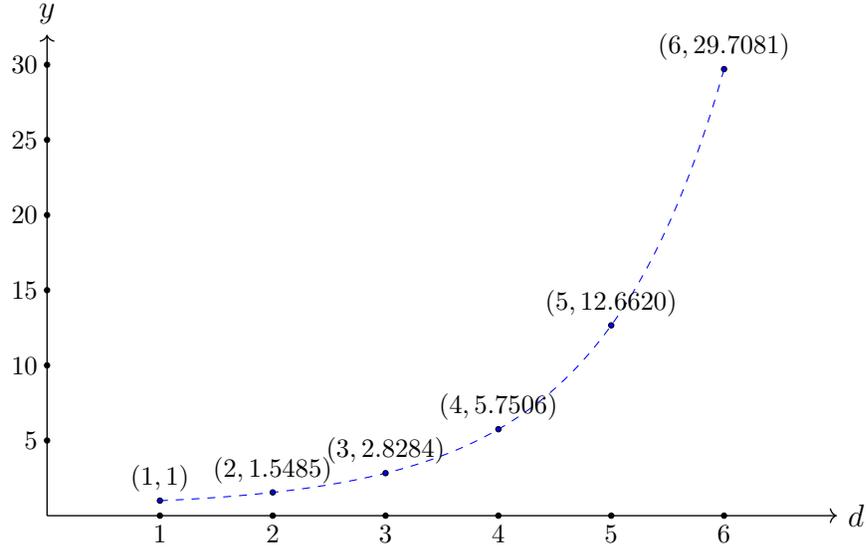
	
	Throughout the paper, we use a very important theorem connecting the $s$ dimensional measure with the induced probabilistic measure, which requires stronger condition for open set conditions.

    \begin{definition}[Strong Open Set Condition]
        Let the IFS $\{S_i\}_{i=1}^m$ satisfies the open set condition. It is said to satisfy the strong open set condition if there exists a bounded open set $V\subset\R^d$ such that $V\cap E\neq\emptyset$, $S_i(V)\subset V$ for all $i$, and $S_i(V)\cap S_j(V) = \emptyset$ while $d\big(S_i(V),S_j(V)\big) > 0$ for all $i\neq j$. Specifically, we use the distance (not a metric) here as:
        \begin{equation*}
            d(X,Y) := \inf\left\{d(x,y):x\in X, y\in Y\right\},
        \end{equation*}
        for all $X,Y\subset \mathbb{R}^d$.
    \end{definition}
    
    With the strong open set condition, in Theorem 3.3 in~\cite{Moran}, the main result is:
    
	\begin{theorem}\label{thm::cvx-set}
		Given a self similar set $E\subset\R^d$ satisfying the strong open set condition. Let $s$ be the Hausdorff dimension of $E$, then:
		\begin{equation}
			\CH^s(E) = \min\left\{\frac{|A|^s}{\mu(A)}: A \mbox{ is convex and } c\leq |A| \leq |E|\right\}, \label{eq::convexinf}
		\end{equation}
        where $\mu$ is the probability measure defined as:
		\begin{equation*}
			\mu(A) = \frac{\CH^s(A\cap E)}{\CH^s(A)},
		\end{equation*}
        and $c:=\min_{i\neq j} d\big(S_i(E), S_j(E)\big)$.
	\end{theorem}
	
	Notice that we have $\min$ on all possible sets, so the value is achievable, and the set satisfying the minimum is special, and we can introduce the idea of an optimal set.\\
	
	\begin{definition}[Optimal Set]
		The set $A$ in Theorem \ref{thm::cvx-set} that achieves the minimum in equation (\ref{eq::convexinf}) is called an optimal set.
	\end{definition}
	
	
	\section{Main results on Upper Bounds}
	
	The discussion on the upper bound is relatively more straightforward since, by the definition of the upper bounds with \textit{minimum}, as long as we can find a set, it automatically casts an upper bound.\\
	
	Through the computation, we exploit the symmetry with the Cantor sets to approximate $\mu(A)$ for any potential
	optimal set $A$. Specifically, we consider a $d$-dimensional sphere centered at $(1/2,1/2,\cdots, 1/2)$ and varying the radius
	radius whose circumference exactly touches one of the corner closest to $(0,0,\cdots, 0)$ of one of the fundamental part on some $k$-th depth. Here, we provide the pseudo-code of the experiment procedure.
	
	\begin{algorithm}
		We use the following procedure to compute the upper bound of the Hausdorff measure with $\CC^d$ up to depth $k$.\\
        Due to symmetry about all $x_i = \frac{1}{2}$ planes, our simulation shall just focus on $\CC_L=\CC\cap[0,\frac{1}{2}]$ and its produced. \rm \label{algo::upper-bdd}
		\begin{algorithmic}[1]
			\State We start with the set $S = \{0\}$, then we iteratively apply the iterated function system to $S$ with a desired number of iteration (as our threshold $k$) so we can get an approximation of left endpoints of remaining intervals of $S \approx \CC_L$.
			\State Generate the lattice of the coordinates $S^d$ by iterating all coordinates.
			\State For all the points in $S^d$, compute its squared distance from $(\frac{1}{2},\frac{1}{2},\cdots, \frac{1}{2})\in \R^d$.
			\State Given a fixed distance $r$, for each squared distance, count the number of lattice of coordinates that is within $r$. This is a lower bound of how much basic cubes are covered on this interval up to this threshold, record this number as $N$. 
            \State Hence, we can compute the lower bound of the probability measure as:
            \begin{equation*}
                \mu(A) \geq \frac{N}{2^{kd}}.
            \end{equation*}
            \State This hence leads to the upper bound of the optimal set using equation (\ref{eq::convexinf}).
		\end{algorithmic}
        As a side note, all computations in the first 4 steps utilizes the \texttt{Fraction} class in the programming language to prevent rounding errors that could impact the accuracy of $N$.
	\end{algorithm}
	
	\begin{remark}
		Algorithm \ref{algo::upper-bdd} is a modification and generalization from~\cite{Deng} which deals with the two-dimensional case, which also trade-off some complexity for a more precise result.
		\begin{itemize}
			\item Given a fixed depth $k$ and a fixed dimension $d$, $S$ is consisted of $2^k$ elements and $S^d$ will be of $m:=2^{dk}$ elements. The computation of the squared distances of each coordinate will be $\CO(d)$, so the total time of computation is $\CO(dm)$ and the time of counting the number of lattice can be done in $\CO(m\log m)$ time, so the total computation time is $\CO(dk2^{dk})$.
			\item In the computation, the lower Cartesian products can be computed within reasonable time ($\leq 10$ minutes) with depth up to 12, but we eventually can only get to a depth of 2 for higher dimension.
		\end{itemize}
	\end{remark}
	
	With some use of computer simulation, we have obtained the upper bounded of the following Cartesian products of the Cantor sets.
	
	\begin{theorem}[Upper Bounds on Measure of Cartesian products of the Cantor sets] \label{thm::upperbound}
		\begin{align*}
			&\CH^{s_2}(\CC^2) \leq \hspace{5.25pt}1.500886049123709,\\
			&\CH^{s_3}(\CC^3) \leq \hspace{5.25pt}2.352741546983966,\\
			&\CH^{s_4}(\CC^4) \leq \hspace{5.25pt}4.089697707421688,\\
			&\CH^{s_5}(\CC^5) \leq \hspace{5.25pt}7.502183963990683,\\
			&\CH^{s_6}(\CC^6) \leq 14.810000552236708,\\
			&\CH^{s_7}(\CC^7) \leq 31.501011683100224,\\
			&\CH^{s_8}(\CC^8) \leq 67.52795132236503.
		\end{align*}
	\end{theorem}
	
	These computations of upper bound can be shown with respect to the na\"ive upper bound:
	
	\begin{figure}[H]
		\begin{center}
			\begin{tikzpicture}[x=1.5cm, y=0.2cm]
				\draw[->, line width=0.5pt] (0,0) -- (8,0) node[right] {$d$};
				\draw[->, line width=0.5pt] (0,0) -- (0,32) node[above] {$y$};
				
				\draw[color=blue, variable=\x, samples=100, domain=1:6, dashed] plot ({\x}, {\x^( (\x * ln(2)/ln(3)) / 2 )});
				\footnotesize
				\draw[fill=blue] (1,1) circle (1pt);
				\draw[fill=blue] (2,1.54856265263) circle (1pt);
				\draw[fill=blue] (3,2.82842712475) circle (1pt);
				\draw[fill=blue] (4,5.75062600477) circle (1pt);
				\draw[fill=blue] (5,12.6620035654) circle (1pt);
				\draw[fill=blue] (6,29.7081993809) circle (1pt);
				\draw[fill=red, dashed] (1,1) circle (1pt) -- (2, 1.500886049123709) circle (1pt) -- (3,2.352741546983966) circle (1pt)  -- (4,4.089697707421688) circle (1pt)  -- (5,7.502183963990683) circle (1pt)  -- (6,14.810000552236708) circle (1pt)  -- (7,31.501011683100224) circle (1pt);
				
				\draw[fill=black] (1,0) circle (1pt) node[below] {1};
				\draw[fill=black] (2,0) circle (1pt) node[below] {2};
				\draw[fill=black] (3,0) circle (1pt) node[below] {3};
				\draw[fill=black] (4,0) circle (1pt) node[below] {4};
				\draw[fill=black] (5,0) circle (1pt) node[below] {5};
				\draw[fill=black] (6,0) circle (1pt) node[below] {6};
				\draw[fill=black] (7,0) circle (1pt) node[below] {7};
				
				\draw[fill=black] (0,5) circle (1pt) node[left] {5};
				\draw[fill=black] (0,10) circle (1pt) node[left] {10};
				\draw[fill=black] (0,15) circle (1pt) node[left] {15};
				\draw[fill=black] (0,20) circle (1pt) node[left] {20};
				\draw[fill=black] (0,25) circle (1pt) node[left] {25};
				\draw[fill=black] (0,30) circle (1pt) node[left] {30};
			\end{tikzpicture}
		\end{center}
		\caption{Computed Improvements on the Upper Bound (\textcolor{red}{red}) compared to the Na\"ive Upper Bound (\textcolor{blue}{blue}).}
	\end{figure}
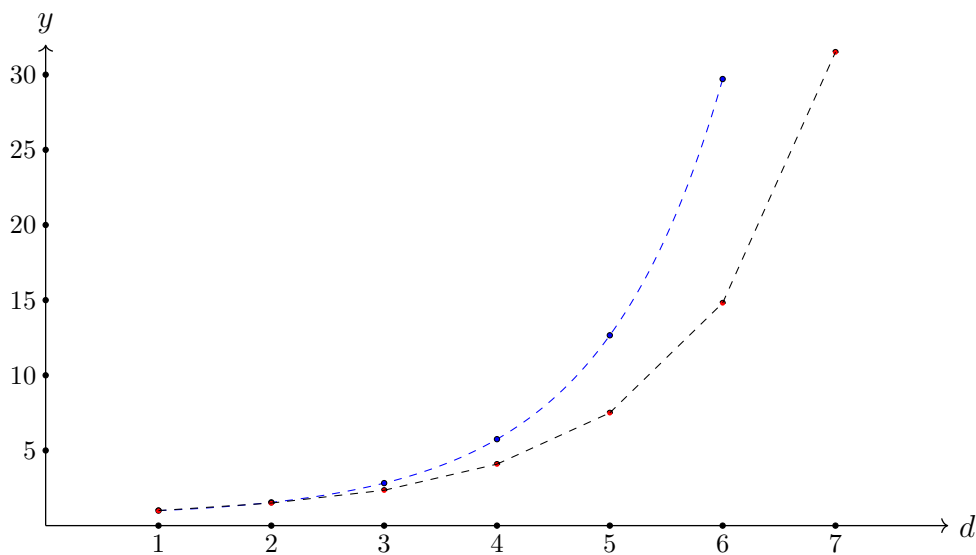
	
	
	\section{Foundations to the Lower Bounds}
	
	\begin{remark}
		With the definition of Quarants and enumeration of the IFS, we can enumerate the basic sets of the $d$-th dimensional and level $1$ Cantor set by inputting $D:=[0,1]^d$ into the $S_i$ function for $i=1,2,\cdots,2^d$.\\

        Additionally, for a string of length $k$, $(i_1,i_2,\cdots, i_k)\in\{1,2,\cdots, 2^d\}^k$, we consider the contraction as the composition of contractions, \ie:
		\begin{equation*}
			S_{i_1,i_2,\cdots, i_k} := S_{i_1}\circ S_{i_2}\circ\cdots\circ S_{i_k}
		\end{equation*}
		Since our main concern is the contractor of the set that we start from $D:=[0,1]^d$, we denote:
		\begin{equation*}
			D_{i_1,i_2,\cdots, i_k} := S_{i_1,i_2,\cdots, i_k}(D).
		\end{equation*}
	\end{remark}

	By definition, $\mu$ is the probability measure so we have $\mu(B)\leq 1$ naturally. Therefore, if we would want to find some tighter lower bound, we can improve by computing a better lower bound on the diameter of the optimal set.\\
	
    By definition, we can develop the following lemma.
    
	\begin{lemma}
		  An optimal set $B\subset\R^d$ must intersect at least two basic sets on level 1, \ie, $B$ intersect at least two sets from $\{D_1,D_2,\cdots,D_{2^d}\}$, thus $|B|\ge1/3=\min\{d(S_i(D),S_j(D):i\ne j\}$.
	\end{lemma}

    This is directly by \ref{thm::cvx-set}, while one can prove this otherwise by scaling the one intersected basic square to the whole domain.
    
	\begin{theorem}\label{thm::5/9diam}
		Suppose $B\subset\R^d$ is the optimal set of $\CC^d$, $|B| > \frac{5}{9}$.
	\end{theorem}
	
	\begin{proof}
		Suppose that $B$ intersects 2 fundamental sets in level $1$, without loss of generality, we may assume that $B$ intersects $D_1$ and $D_2$. For the sake of contradiction, we suppose that $B\cap\big((\bigcup_{x:\text{odd}}D_{1,x})\cup(\bigcup_{y:\text{even}}D_{2,y})\big) = \emptyset$.\\
		Then, we construct a scaling and translation map $\varphi$ such that for all odd $x$ and even $y$, it maps:
		\begin{equation*}
			D_{1,y}\mapsto D_y \mbox{ and } D_{2,x}\mapsto D_x.
		\end{equation*}
        Concretely, we can define the map $\varphi:D_{1}\cup D_2\to \CC^d$ as follows:
        \begin{equation*}
            \varphi(x) = \begin{cases}
                3x, & \mbox{ when } x\in D_1,\\
                3\big(x - (\frac{2}{3},0,0,\cdots, 0)\big), & \mbox{ when } x\in D_2.
            \end{cases}
        \end{equation*}
		Then, we consider sets $T_1,T_2,\cdots, T_{2^d}$ constructed by:
		\begin{equation*}
			T_j = \begin{cases}
				\varphi(D_{1,j}\cap B) \mbox{ when } j \mbox{ is even},\\
				\varphi(D_{2,j}\cap B) \mbox{ when } j \mbox{ is odd}.
			\end{cases}
		\end{equation*}
		Now, denote $T:= \bigcup_{i=1}^{2^d}T_i$. Note that the scaling is exactly by 3 and is the exact duplicate of the self similar sets, we have:
		\begin{equation*}
			\mu(T) = 2^d\mu(B).
		\end{equation*}
		Now, we consider that the scaling factor is by $3$, so we have:
		\begin{equation*}
			|T| \leq 3|B|.
		\end{equation*}
		However, we shall not be satisfactory with the inequality. We need a strict inequality. We observe that for the $\varphi$ map, if we break into each dimension, they:
		\begin{itemize}
			\item Either has the mapping with a negative translation factor of $-\frac{2}{3}$ and $-\frac{4}{3}$, or
			\item has the maximum possible distance in that that dimension smaller than $\frac{1}{3}$.
		\end{itemize}
		Hence, either way, the diameter of the shape after the transition would be strictly bounded by 3 times of the original diameter. Thus, $|T| < 3|B|$.\\
		Now, if we use $T$ as as the optimal set in equation \ref{eq::convexinf}:
		\begin{equation*}
			\frac{|T|^s}{\mu(T)} < \frac{3^s|B|^s}{\mu(T)} = \frac{2^d|B|^s}{2^d\mu(B)} = \frac{|B|^s}{\mu(B)},
		\end{equation*}
		which implies that $T$ is smaller than $B$, and thus $B$ cannot be the optimal set, which contradicts with our assumption that $B$ is optimal set.\\
		Therefore, $B$ must intersect at least one of the sets on the corner, specifically, $B\cap\big((\bigcup_{x:\text{odd}}D_{1,x})\cup(\bigcup_{y:\text{even}}D_{2,y})\big) \neq \emptyset$, and since it already intersects $D_1$ and $D_2$, so it must at least intersect some $D_{2,y}$ with $D_1$ or some $D_{1,x}$ with $D_2$. Thus, such distance is guaranteed to be $\frac{5}{9}$.
	\end{proof}
	
	For the better approximation procedure, we shall consider the repulsive lemma and repulsive pairs.
	
	\begin{definition}[Repulsive Pairs]
		$\{U, V\}$ is defined to be a repulsive pair of order $k$ with distance $d$ if $U$ and $V$ are basic sets of level $k$ for $\CC^d$ and $d_{\rm H}(U,V) > d$.
	\end{definition}
	
	\begin{lemma}[Repulsive Lemma]
		Suppose $\{U,V\}$ is a repulsive pair of order $k$ with distance $|B|$, then:
		\begin{equation*}
			\mu\big(B\cap(U\cup V)\big)\leq 2^{-dk}.
		\end{equation*}
	\end{lemma}
	
	\begin{proof}
		Suppose $U = S_I(D)$ and $V = S_J(D)$, where $I,J\in\{1,2,\cdots, 2^d\}^k$. Since $|B| < d_{\rm H}(U,V)$, there exists $m\in \ZZ^+$ such that for all $L\in\{1,2,\cdots, 2^d\}^m$:
		\begin{equation*}
			d_{\rm H}\big(B\cap \big(S_{L,I}(D) \cup S_{L,J}(D)\big)\big).
		\end{equation*}
		Hence, by the monotonicity of the measure:
		\begin{align*}
			\mu\big(B\cap (U\cup V)\big) & \leq \mu\left(\bigcup_{L\in\{1,2,\cdots, 2^d\}^m}\big(B\cap \big(S_{L,I}(D) \cup S_{L,J}(D)\big)\big)\right)\\
			& = (2^d)^m\cdot 2^{-d(k+m)} = 2^{-dk}.
		\end{align*}
	\end{proof}
	
	Now, we can use this notation to generalize a tighter lower bound by using the following proposition.
	
	\begin{proposition}\label{thm::upbd-prob-msre}
		Suppose $B\subset \R^d$ is an optimal set of $\CC^d$ such that $|B|\leq L$, then:
		\begin{equation}
			\mu(B) \leq 1 - N_k(L)2^{d-kd-1}, \label{eq::upbd-prob-msre}
		\end{equation}
		where $N_k(L)$ is the maximum number of matching (vertices that do not share a common edge) over all edges that connect between:
		\begin{align*}
			I\in& X_k = \{1,i_2,\cdots, i_k: i_t = 1,2,\cdots, 2^d \mbox{ for } 2\leq t\leq k\},\\
			J\in& X_k = \{2^d,i_2,\cdots, i_k: i_t = 1,2,\cdots, 2^d \mbox{ for } 2\leq t\leq k\},
		\end{align*}
		in which $|S_I(0) - S_J(0)| > L$, where $0:=(0,\cdots,0)\in\R^d$.
	\end{proposition}
	
	\begin{proof}
		Consider that there is a total of $N_k(L)$ disjoint repulsive pairs of order $k$ for $[0,\frac{1}{2}]^d\cup[\frac{1}{2},1]^d$ whose Hausdorff distances are bounded below by $L$, and note that the diameter of $B$ is bounded above by $L$ so the optimal set covers at most one of those fundamental sets there, hence:
		\begin{equation*}
			\mu\left(B\cap \left(\left[0,\frac{1}{2}\right]^d\cup\left[\frac{1}{2},1\right]^d \right)\right) \leq \frac{1}{2^{d-1}} - \frac{N_k(L)}{2^{kd}}.
		\end{equation*}
		Now, by symmetry, there will be a total of $2^{d-1}$ such pairs of quadrants, so we have:
		\begin{equation*}
			\mu(B) \leq 1 - 2^{d-1}\cdot\frac{N_k(L)}{2^{kd}} = 1 - N_k(L)2^{d-kd-1}. \ms
		\end{equation*}
	\end{proof}
	
	\begin{remark}
		The computation of $N_k(L)$ is rather costly, especially when we have to deal with different dimensions and different thresholds. However, we may use a bound of $N_k(L)$ as perfect matching or no matching at all:
		\begin{equation*}
			0 \leq N_k(L) \leq 2^{(k-1)d},
		\end{equation*}
		and by plugging it in equation \ref{eq::upbd-prob-msre}, we obtain that:
		\begin{equation*}
			\mu(B) \leq 1,
		\end{equation*}
		which aligns with the maximum possible probability measure.
	\end{remark}
	
	Furthermore, we will use Proposition \ref{thm::upbd-prob-msre} to define a new function.
	
	\begin{definition}
		Let $k\in\Z^+$ and $d\in\Z^+$ be fixed, we define the function $W_{d,k}:\R^+\to\R$ as:
		\begin{equation*}
			W_{d,k}(x) = \big(1-N_k(x)\cdot 2^{d-kd-1}H_d\big)^{1/s_d},
		\end{equation*}
		where $N_k(x)$ is defined as of in Proposition \ref{thm::upbd-prob-msre}, $H_d$ is the upper bound of Hausdorff dimension of $\CC^d$ and $s_d$ is the Hausdorff dimension of $\CC^d$.
	\end{definition}
	
	\begin{proposition}
		Here, we can give some properties of the $W_{d,k}$ function:
		\begin{enumerate}
			\item For fixed $d$ and $k$, $W_{d,k}$ is monotonic, \ie, $W_{d,k}(x) \leq W_{d,k}(y)$ if $x\leq y$.
			\item If $W_{d,k}(c) < c$, then $|B|\neq c$, where $B$ is an optimal set of $\CC^d$.
		\end{enumerate}
	\end{proposition}
	
	\begin{proof}
		\begin{enumerate}
			\item The monotonicity is by the fact that $N_k(x) \geq N_k(y)$ when $x \leq y$, since there will be at least that many repulsive pairs when the distances are shorter.
			\item For the sake of contradiction, suppose $W_{d,k}(c) < c$ and $|B| = c$, with Proposition \ref{thm::upbd-prob-msre}, we have:
			\begin{align*}
				\mu(B) \leq 1 - 2N_{k}(c)4^{-k} < \frac{c^s}{H_d} \leq \mu(B),
			\end{align*}
			which is a contradiction.
		\end{enumerate}
	\end{proof}
	
	Note that the $W_{d,k}(x)$ function allows more possibilities in finding stricter bounds for the Hausdorff measure. However, it is very much constrained by the computation capabilities of the computer. When the dimension $d$ and/or the layer $k$ gets larger, it requires exponentially longer time to compute the number of repulsive pairs.
	
	\section{Computation of Lower Bounds}
	
	With a few lower dimensions, however, we may utilize the feature on very low levels to computer better lower bounds on the Hausdorff measure.\\
	
	For dimension 3, recall from Theorem \ref{thm::upperbound}, we have:
	\begin{equation*}
		\CH^{s_3}(\CC^3) \leq 2.352741546983966.
	\end{equation*}
	Also, by Theorem \ref{thm::5/9diam}, let $B$ be the optimal set, we know that $|B| > 5/9$, so we can consequently have:
	\begin{equation*}
		\mu(B) \geq \frac{(5/9)^{3\log_32}}{2.352741546983966} > 0.139716 > \frac{1}{8}.
	\end{equation*}
	Now, we can claim that $|B| \geq\frac{2\sqrt{11}}{9}$, which happens to be the Hausdorff distance between the following eight repulsive pairs or order 2:
	\begin{align*}
		\{D_{1,1}, D_{2,7}\}, \{D_{1,2}, D_{2,8}\}, \{D_{1,3}, D_{2,5}\}, \{D_{1,4}, D_{2,5}\},\\
		\{D_{1,5}, D_{2,4}\}, \{D_{1,6}, D_{2,3}\}, \{D_{1,7}, D_{2,1}\}, \{D_{1,8}, D_{2,2}\}.
	\end{align*}
	That is,
	\begin{equation*}
		d_{\rm H}(D_{1,1},D_{2,7}) = \cdots = d_{\rm H}(D_{1,8},D_{2,2}) = \sqrt{\frac{2^2 + 2^2 + 6^2}{9^2}} = \frac{2\sqrt{11}}{9}.
	\end{equation*}
	Hence, if $|B| < \frac{2\sqrt{11}}{9}$, there will be 8 repulsive pairs and
	\begin{equation*}
		\mu(B)\leq 8\times \frac{1}{64} = \frac{1}{8},
	\end{equation*}
	which contradicts the fact that $\mu(B) > \frac{1}{8}$, implying $|B|\geq\frac{2\sqrt{11}}{9}$.\\
    
	With the updated diameter, we have:
	\begin{equation*}
		\mu(B) \geq \frac{(2\sqrt{11} / 9)^{3\log_32}}{2.352741546983966} > 0.238561 > \frac{12}{64} = \frac{3}{16}.
	\end{equation*}
	Now, we claim that $|B| \geq \frac{2\sqrt{2}}{3}$, which happens to be the Hausdorff distance between the following four repulsive pairs of order 2:
	\begin{equation}
		\{D_{1,1}, D_{2,8}\}, \{D_{1,3}, D_{2,6}\}, \{D_{1,5}, D_{2, 4}\}, \{D_{1,7}, D_{2,2}\}, \label{ex:rep-pairs}
	\end{equation}
	which we have:
	\begin{equation*}
		d_{\rm H}(D_{1,1},D_{2,8}) = \cdots = d_{\rm H}(D_{1,7},D_{2,2}) = \sqrt{\frac{2^2 + 2^2 + 8^2}{9^2}} = \frac{2\sqrt{2}}{3}.
	\end{equation*}
	Hence, if $|B| < \frac{2\sqrt{2}}{3}$, there will be 4 repulsive pairs and
	\begin{equation*}
		\mu(B)\leq (8 + 4)\times \frac{1}{64} = \frac{3}{16},
	\end{equation*}
	which contradicts $\mu(B) > \frac{3}{16}$, implying $|B|\geq\frac{2\sqrt{2}}{3}$. 
    Again, with the updated diameter, we have:
	\begin{equation*}
		\mu(B) \geq \frac{(2\sqrt{2} / 3)^{3\log_32}}{2.352741546983966} > 0.380202,
	\end{equation*}
	and this turns out to be the best approximations give the upper bound and level 2.\\
	
	Then, we can think of some consequences of such result.
	
	\begin{lemma}
		An optimal set of $\CC^3$ must intersect four basic sets of level 1.
	\end{lemma}
	
	\begin{proof}
		First, for the sake of contradiction, we suppose that all optimal sets of $\CC^3$ covers at most two basic sets of level 1.\\
		However, from \eqref{ex:rep-pairs}, we found repulsive pairs such that $\mu(B) \geq 0.380202 > \frac{3}{8}$, there must have been an optimal set that intersects at least three of the basic sets.\\
		Then, suppose that all optimal sets of $\CC^3$ covers at most three basic sets of level 1.\\
		Similarly, we can suppose that it intersects $D_1$, $D_2$, and $D_3$. We can use \eqref{ex:rep-pairs} from $D_1$ and $D_2$, and we can assume all basic sets of level 2 from $D_3$ are included. Hence, we have $\mu(B) > 0.380202 > \frac{3}{8}$, there must have been an optimal set that intersects at least four of the basic sets.
	\end{proof}
	
	Since we now have four basic sets of level 1 includes, we can try to find better bounds from that.\\
	
	With four basic sets of level 1 included, we claim that $|B|\geq\frac{2\sqrt{26}}{9}$ where we have to consider two cases, without loss of generality: 
	\begin{itemize}
		\item When the intersection is $D_1,D_2,D_3,D_4$, or the four basic sets belong to the same face.\\
		Here, this aligns with the Hausdorff distance between the following eight repulsive pairs of order 2:
		\begin{align}
			\{D_{1,1}, D_{4,7}\}, \{D_{1,2}, D_{4,8}\}, \{D_{1,3}, D_{4,5}\}, \{D_{1,4}, D_{4,5}\}, \notag\\ 
			\{D_{1,5}, D_{4,4}\}, \{D_{1,6}, D_{4,3}\}, \{D_{1,7}, D_{4,1}\}, \{D_{1,8}, D_{4,2}\}. \label{eq::pair4-1}
		\end{align}
		\item When the intersections are $D_1,D_2,D_3,D_5$, or the four basic sets spans all faces.\\
		Here, we consider the Hausdorff distance between the following eight repulsive pairs of order 2:
		\begin{align}
			\{D_{2,1}, D_{5,7}\}, \{D_{2,2}, D_{5,8}\}, \{D_{2,3}, D_{5,6}\}, \{D_{2,4}, D_{5,5}\}, \notag\\
			\{D_{2,6}, D_{3,1}\}, \{D_{2,8}, D_{3,2}\}, \{D_{3,3}, D_{5,1}\}, \{D_{3,4}, D_{5,2}\}. \label{eq::pair4-2}
		\end{align}
	\end{itemize}
	Specifically, here we have:
	\begin{align*}
		d_{\rm H}(D_{1,1},D_{4,7}) &= \cdots = d_{\rm H}(D_{3,8},D_{2,2}) = d_{\rm H}(D_{2,1},D_{5,7}) = \cdots = d_{\rm H}(D_{3,4},D_{5,2})\\
		&= \sqrt{\frac{2^2 + 6^2 + 8^2}{9^2}} = \frac{2\sqrt{26}}{9}.
	\end{align*}
	Hence, by if $|B| < \frac{2\sqrt{26}}{9}$, there will be (at least) eight repulsive pairs (for both cases), and hence:
	\begin{equation*}
		\mu(B)\leq (16+8)\times \frac{1}{64} = \frac{3}{8},
	\end{equation*}
	which is a contradiction to that $\mu(B) > \frac{3}{8}$, which implies that $|B|\geq\frac{2\sqrt{26}}{9}$.\\
	With the updated diameter, we have:
	\begin{equation*}
		\mu(B) \geq \frac{(2\sqrt{26} / 9)^{3\log_32}}{2.352741546983966} > 0.538462 > \frac{1}{2}.
	\end{equation*}
	
	Now, we may use this as the conclusion to a stronger result:
	
	\begin{proposition}
		The optimal set of $\CC^3$ that covers five basic sets of level 1.
	\end{proposition}
	
	\begin{proof}
		For the sake of contradiction, we suppose that all optimal sets of $\CC^3$ covers at most four basic sets of level 1.\\
		However, from \eqref{eq::pair4-1} or \eqref{eq::pair4-2}, we found repulsive pairs such that $\mu(B) \geq 0.538462 > \frac{1}{2}$, there must have been an optimal set that intersects at least five of the basic sets.
	\end{proof}
	
	\begin{proposition}
		$\CH^{s_3}(\CC^3) > (2\sqrt{26} / 9)^{3\log_32} \geq 1.811621$.
	\end{proposition}
	
	\begin{proof}
		This is directly from the lower bound of $|B| \geq \frac{2\sqrt{26}}{9}$ and $\mu(B) \leq 1$.
	\end{proof}
	
	\begin{remark}
		For this, we can quite trivially note that:
		\begin{equation*}
			\CH^{s_3}(\CC^3) > 1.8 > \CH^{s_2}(\CC^2).
		\end{equation*}
	\end{remark}

    From this, we can conclude that the Hausdorff measure of the third product of the ternary Cantor set has higher $2\log_32$-dimensional Hausdorff measure compared to the second product of the ternary Cantor set with respect to the $\log_32$-dimensional Hausdorff measure.
    
	
	\printbibliography

@book{Falconer,
    author = {Falconer, Kenneth},
    title = {Fractal Geometry, Mathematical Foundations and Applications},
    publisher = {University of St Andrews},
    year = {2003}
}

@book{FalconerII,
	author = {Falconer, Kenneth},
	title = {The Geometry of Fractal Sets},
	publisher = {Cambridge University Press},
	year = {1985}
}

@article{Deng,
author = {Deng, Juan and Rao, Hui and Wen, Zhi-ying},
title = {Hausdorff Measure of Cartesian Product of the Ternary Cantor Set},
journal = {Fractals},
volume = {20},
number = {01},
pages = {77-88},
year = {2012},
doi = {10.1142/S0218348X12500077},
URL = {https://doi.org/10.1142/S0218348X12500077},
eprint = {https://doi.org/10.1142/S0218348X12500077}
}

@article{Moran,
	author = {Llorente, Marta and Mor\'an, Manuel},
	title = {Self-similar sets with optimal coverings and packings},
	journal = {Journal of Mathematical Analysis and Applications},
	volume = {334},
	pages = {1088-1095},
	year = {2007},
}

@article{Hausdorff1919,
	author = {Hausdorff, F.},
	journal = {Mathematische Annalen},
	pages = {157-179},
	title = {Dimension und äußeres Maß},
	url = {http://eudml.org/doc/158784},
	volume = {79},
	year = {1919},
}

@article{MESURESDE,
    author = {Jacques Marion},
    title = {Mesuresde Hausdorff D'Ensembles Fractals},
    journal = {Ann. SC. math. Québec},
    url ={https://www.labmath.uqam.ca/~annales/volumes/11-1/PDF/111-132.pdf},
    year = {1987}
}
	
\end{document}